\documentclass[reqno,12pt]{amsart}

\if
\else
\fi
\newlength{\insidemargin}
\newlength{\outsidemargin}
\input xy
\xyoption{all}
\usepackage{graphicx}
\usepackage{transparent}
\usepackage[mathscr]{eucal}
\usepackage{amsmath, amssymb,amsfonts,bbm,stmaryrd,hhline}
\usepackage{url}
\usepackage[breaklinks]{hyperref}
\usepackage[T1]{fontenc}
\usepackage{wrapfig}
\usepackage{color} 
\definecolor{darkred}{rgb}{1,0,0} 
\definecolor{darkgreen}{rgb}{0,0.8,0}
\definecolor{darkblue}{rgb}{0,0,1}

\hypersetup{colorlinks, linkcolor=darkblue, filecolor=darkgreen,
	urlcolor=darkred, citecolor=darkgreen}

\usepackage{lastpage}
\usepackage{fancyhdr}
\usepackage{etoolbox}
\makeatletter
\patchcmd{\@settitle}{\uppercasenonmath\@title}{}{}{}
\patchcmd{\@setauthors}{\MakeUppercase}{}{}{}
\patchcmd{\section}{\scshape}{}{}{}
\makeatother

\numberwithin{equation}{section}

\newtheorem{theorem}{Theorem}[section] 
\newtheorem{lemma}[theorem]{Lemma}     

\newtheorem{example}[]{Example}
\newtheorem{remark}[]{Remark}

\def    \eps    {\epsilon}

\newcommand{\CA}{{\mathcal A}}

\newcommand{\CM}{{\mathcal M}}

\newcommand{\CS}{{\mathcal S}}

\newcommand{\A}{{\mathcal A}}

\newcommand{\D}{{\Delta}}

\def    \F      {{\mathbb F}}

\def    \C      {{\mathbb C}}

\def    \reals      {{\mathbb R}}
\def    \Z      {{\mathbb Z}}
\def    \N      {{\mathbb N}}

\def    \T      {{\mathbb T}}
\def    \CP     {{\mathbb C}{\mathbb P}}

\def    \12    {{\frac{1}{2}}}

\def    \p      {\partial}

\def    \im     {\operatorname{im}}

\def    \Sp     {\operatorname{Sp}}

\def    \MUCZ  {\operatorname{\mu_{\scriptscriptstyle{CZ}}}}

\title[On Non-contractible Periodic Orbits]
 {\large \textnormal{On non-contractible hyperbolic periodic orbits and}\\ \textnormal{periodic points of symplectomorphisms}}

\author{Marta Bator\'eo}
\thanks{ \\
	This work was done while the author was a Post-Doctorate at IMPA, funded by CAPES-Brazil.\\
	2000 {\normalfont\itshape
		Mathematics Subject Classification\/} 53D40 (primary), 37J10, 70H12 (secondary) }

\begin{document}

\maketitle
\vspace*{-0.5cm}
\begin{abstract}
We prove the existence of infinitely many periodic points of symplectomorphisms isotopic to the identity if they admit at least one (non-contractible) hyperbolic periodic orbit and satisfy some condition on its flux. The obtained periodic points correspond to periodic orbits whose free homotopy classes are formed by iterations of the hyperbolic periodic orbit. Our result is proved for a certain class of closed symplectic manifolds and the main tool we use is a variation of Floer theory for non-contractible periodic orbits and symplectomorphisms, the Floer--Novikov theory. 

For a certain class of symplectic manifolds, the theorem generalizes the main results proved for Hamiltonian diffeomorphisms in \cite{GG:hyp12} and for symplectomorphisms and contractible orbits in \cite{Ba_hyp}. 
\end{abstract}

\tableofcontents
\thispagestyle{empty}
\section{Introduction and the main result}

In this paper, we prove that given a non-contractible hyperbolic periodic orbit of a symplectomorphism isotopic to the identity there are infinitely many periodic points provided the symplectomorphism satisfies some constraints on its flux. The result holds for certain symplectic manifolds such as products of some surfaces with complex projective spaces (see Section~\ref{section:mainthm} for more details). 

The main theorem of this paper fits into the conjecture originally (as far as we know) stated by Hofer and Zehnder claiming (see \cite[p.263]{HZ11})

"(...) that every Hamiltonian map on a compact symplectic manifold $(M, \omega)$ possessing more fixed points than necessary required by the V. Arnold conjecture possesses always infinitely many periodic points."

For instance, the expected bound for $\CP^n$ is $n+1$. This conjecture was motivated by the results of Gambaudo and Le Calvez in \cite{Gam_LeCa} and Franks in \cite{Fr88} (see also \cite{Fr92,Fr96}) where they prove that an area preserving diffeomorphism of $S^2$ with more than two fixed points has infinitely many periodic points; see also \cite{BH11,CKRTZ12,Ke12} for symplectic topological proofs. 

More accurately, our result fits in a variant of this conjecture described by G{\"u}rel in \cite{Gu_noncon,Gu_linear}; see also \cite{GG:hyp12}. This version suggests the presence of infinitely many periodic points provided the existence of a fixed point which is \emph{unnecessary} from a homological or geometrical perspective.
In fact, Theorem~\ref{maintheo} asserts that, for a certain class of symplectic manifolds, a symplectomorphism isotopic to the identity with a (non-contractible) hyperbolic periodic orbit possesses infinitely many periodic points, as long as it satisfies some condition on its flux. The theorem is a symplectic analogue of a result proved in \cite{GG:hyp12} for Hamiltonian diffeomorphisms and also generalizes the theorem in \cite{Ba_hyp} for contractible periodic orbits of symplectomorphisms. Moreover, in dimension greater than two, the conjecture for Hamiltonian diffeomorphisms is proved in \cite{Gu_noncon} and \cite{Gu_linear} where the fixed points are assumed to be \emph{unnecessary} from a homological perspective. 

Recall that, in general, a symplectomorphism need not have periodic points, as is the case of $\psi\colon \T^{2n}\rightarrow \T^{2n}$ defined by
\[
\psi(x,y)=(x+\theta,y)\quad\quad\text{with}\quad \theta\not=1,
\]
whereas, for many symplectic manifolds, a Hamiltonian diffeomorphism has an infinite number of periodic points. In fact, Conley conjectured (\cite{Co}) the existence of infinitely many periodic orbits of a Hamiltonian diffeomorphism on the $2$-torus and the statement has now also been proved, e.g. for negative monotone manifolds or symplectic manifolds with ${c_1}|_{\pi_2(M)}=0$; see \cite{CGG11,GG:action09,He12} and also \cite{FrHa03,Gi:conley,GG:conley12,Hi,LeC,SZ92}. Observe that in the conjecture mentioned above described by G\"urel there is an assumption, for instance, on the existence of a periodic orbit of a certain geometrical type. This is in contrast with the Conley conjecture which is unconditional. Hence, owing to the above example we expect some condition in order for a symplectomorphism, which is not necessarily a Hamiltonian diffeomorphism, to have infinitely many periodic points.

\subsection{Existence of infinitely many periodic orbits}\label{section:mainthm}

Consider a closed monotone symplectic manifold $(M^{2n}, \omega)$ with minimal Chern number $c_1^{\min}$ and toroidal minimal Chern number $c_{1,T}^{\min}$ (for the exact definitions see \eqref{eqn:minimalChernnumb} and \eqref{eqn:toroidalc1min}, respectively). We denote by $q$ the generator of the Novikov ring $\Lambda=\mathbb{Z}_2[q^{-1},q]$ normalized with degree $-2c_1^{\min}$ (see Section~\ref{section:loopspaces}). 

Given a homotopy class (with fixed end points) of a symplectic isotopy, $[\phi_t]$, connecting the identity to some symplectomorphism $\phi$, we denote by $\varrho$ the cohomology class on $\mathcal{L}M$ (i.e. the space of loops in $M$)
\[
\varrho\colon=-\overline{[\omega]} + \text{ev}^*(\text{Flux}[\phi_t])\in H^1(\mathcal{L}M;\reals),
\]
where $\overline{[\omega]}$ is the cohomology class in $H^1(\mathcal{L}M;\reals)$ associated with the symplectic form $\omega$ (see Section~\ref{section:loopspaces}), $\text{ev}\colon \mathcal{L}M\rightarrow M$ is the evaluation map $x\mapsto x(0)$ and $\text{Flux}([\phi_t])$ is the image of the class $[\phi_t]$ by the homomorphism (see Section~\ref{section:po_af})

\[\text{Flux}\colon \widetilde{\text{Symp}_0}(M,\omega) \rightarrow H^1(M,\reals); \quad 
[\phi_t] \mapsto \left[ \int_0^1 \omega (X_t, \cdot) dt\right].
\]

The following theorem on periodic orbits of symplectomorphisms is the main result of this paper.
	
	\begin{theorem}\label{maintheo}
		\
		
		Let $(M^{2n}, \omega)$ be a closed, connected, strictly spherically monotone symplectic manifold. Assume that \[c^{\min}_{1,T}\geq \frac{n}{2}+1,\]
		\begin{eqnarray}\label{l:1}\beta*\alpha =q[M] \quad \quad \text{in}\; HQ_*(M)=H_*(M)\otimes \Lambda\end{eqnarray}
		for some ordinary homology classes $\alpha, \beta\in H_*(M)$ with $\deg(\alpha), \deg(\beta)<2n$
		and 
		\begin{eqnarray}\label{item:qs} \text{either} \quad c^{\min}_{1} =c^{\min}_{1,T} \quad \text{or} \quad \deg(\alpha)\geq 3n+1-2c^{\min}_{1,T}.
		\end{eqnarray}	
		
		Then any symplectomorphism $\phi$ in $\text{Symp}_0(M,\omega)$ with
		\begin{enumerate}
			\item a hyperbolic (non-contractible) periodic orbit $\gamma$, with free homotopy class $\zeta$,
			\item $\varrho$ and $\overline{c_1}$ strictly toroidally proportional and strictly $\zeta$-toroidally proportional
\end{enumerate}
		has infinitely many periodic points.

The corresponding periodic orbits lie in the free homotopy classes formed by iterations of the hyperbolic periodic orbit $\gamma$.
	\end{theorem}

The conditions in the hypotheses of this theorem involve both the manifold and the symplectomorphism. 

Among the manifolds meeting the requirements of Theorem~\ref{maintheo} are products of surfaces with genus greater than or equal to two, $\Sigma_{g\geq 2}$, and complex projective spaces (see Examples~\ref{example:firstchern}~and~\ref{example:quantumprod} and Remark~\ref{rmk:exampleManifolds}). These manifolds satisfy the first requirement in \eqref{item:qs}. More generally, instead of $\Sigma_{g\geq 2}$, we may consider closed K\"ahler manifolds with negative sectional curvature (e.g., complex hyperbolic
spaces); see, e.g., \cite{Gu_noncon} for a discussion of these manifolds. 

The hyperbolicity assumption (i) on the orbit $\gamma$ is crucial (see Remark~\ref{rmk:hyp}). Ginzburg and G{\"u}rel (\cite[Section~3]{GG:hyp12}) proved that the energy required for a (\emph{Floer--Novikov}) trajectory to approach a $k$-th iteration of $\gamma$ and crossing its fixed neighborhood is bounded from below by a strictly positive constant independent of $k$. The requirement that $\gamma$ is hyperbolic is so that the orbit has the feature just mentioned and which is described in Section~\ref{section:BCEthm}. Hence, more \emph{generally}, there are infinitely many periodic points if in the hypotheses of Theorem~\ref{maintheo} the hyperbolicity condition is replaced by the property described in Theorem~\ref{thm:ballcrossing}. 

Condition (ii) involves both the manifold and the symplectomorphism $\phi$. It gives constraints on the flux of an isotopy $\{\phi_t\}$ (connecting $id$ and $\phi$). Here, there is a relation between the cohomology classes $\varrho$ and $\overline{c_1}$ on the loop space $\mathcal{L}M$. As explained in the beginning of this section, the class $\varrho$ depends on the cohomology class of the symplectic form $\omega$ and on the flux of the considered isotopy $\{\phi_t\}$ and the class $\overline{c_1}$ depends on the first Chen class of the manifold. There are maps meeting these requirements. Namely, consider on $M=\Sigma_{g\geq 2}\times \CP^n$ a symplectic isotopy $\phi_t=\psi_t\times id \colon M\rightarrow M$ where $\psi_t$ is a symplectic isotopy of $\Sigma_{g\geq 2}$ and $id\colon \CP^n \rightarrow \CP^n$ is the identity map on $\CP^n$. Let $\gamma$ be a (non-contractible) periodic orbit of $\phi_t$ in $M$ with free homotopy class $\zeta$. Assume also that $\text{Flux}([\phi_t])\in H^1(M;\reals)$ and $\overline{[\omega]}\in H^1(\mathcal{L}M;\reals)$ are positively proportional on $\pi_1(\mathcal{L}M)$ and on $\pi_1(\mathcal{L}_{\zeta}M)$, i.e.

\[
\text{ev}^*\text{Flux}([\phi_t])|_{\pi_1(\mathcal{L}M)}=\lambda_1 \overline{[\omega]}|_{\pi_1(\mathcal{L}M)} \quad \text{and} \quad
\text{ev}^*\text{Flux}([\phi_t])|_{\pi_1(\mathcal{L}_{\zeta}M)}=\lambda_2 \overline{[\omega]}|_{\pi_1(\mathcal{L}_{\zeta}M)} 
\]
for some positive constants $\lambda_1,\lambda_2>0$. Since $M$ is strictly spherically and toroidally monotone, such an isotopy $\phi_t$ satisfies condition (ii). (Cf. \cite[Proposition~1.2 and Example~1]{Ba_hyp} for the case where $\gamma$ is contractible.)

Moreover, if the periodic orbit $\gamma$ is contractible, then, when the manifold $M$ is toroidally monotone and $c^{\min}_{1} =c^{\min}_{1,T}$, Theorem~\ref{maintheo} generalizes the main results in \cite{Ba_hyp} and \cite{GG:hyp12}.

\subsection{Acknowledgments} 
The author is grateful to Viktor Ginzburg for valuable discussions. Part of this work was carried out at ICMAT. The author would like to thank the institute for their warm hospitality and support.

\section{Preliminaries}\label{section:prelim}
	
In this section, we introduce the notation used throughout the paper and review some facts needed to prove the results.
	\subsection{Symplectic manifolds}
	\label{section:sympmnfd}
	Let $(M^{2n},\omega)$ be a \emph{closed} (i.e. compact and with no boundary) symplectic manifold and consider an almost complex
	structure $J$ on $TM$ compatible with $\omega,$ i.e. such that $g(X, Y) := \omega(X, JY)$ is a Riemannian metric on $M.$
	
	Since the space of almost complex structures compatible with $\omega$ is connected, the first Chern class $c_1\in H^2(M;\Z)$ is uniquely determined by $\omega$.
	The \emph{minimal Chern number} of a symplectic manifold $(M,
	\omega)$ is the positive integer $c_1^{\min}$ which generates the discrete group formed by the integrals
	of $c_1$ over the spheres in $M$, i.e.,
	\begin{eqnarray}\label{eqn:minimalChernnumb}
	\left<c_1, \pi_2(M)\right> = c_1^{\min} \Z
	\end{eqnarray}
	where $c_1^{\min}\in\N$ with the convention $c_{1}^{\text{min}}=\infty$ if the image of $c_1\colon \pi_2(M)\xrightarrow{} \Z$ is $0$. 
	
	The symplectic manifold $(M,\omega)$ is called \emph{spherically monotone} if the integrals of the cohomology classes $c_1$ and
	$[\omega]$ over spheres satisfy the proportionality condition
	\[
	[\omega]|_{\pi_2(M)}=\lambda_{S}   {c_1}|_{\pi_2(M)},
	\] 
	for some non-negative constant $\lambda_{S}\in\reals.$ If $[\omega]|_{\pi_2(M)}$ and ${c_1}|_{\pi_2(M)}$ are non-zero, then we say that $(M,\omega)$ is \emph{strictly spherically monotone}. Observe that, in this case, $c_1^{\min}<\infty$.  
	
\subsection{Loop spaces}\label{section:loopspaces} In this section, we follow \cite{BH01_non-con}, \cite{O_flux06} and references there in. 
 
Let $\mathcal{L}M:=\mathcal{C}^{\infty}(S^1,M)$ be the space of smooth free loops in $M$ where $S^1=\reals / \Z$. The first Chern class $c_1\in H^2(M;\Z)$ of $(M,\omega)$ defines a cohomology class 
\[
\overline{c_1}\in H^1(\mathcal{L}M;\Z)=\text{Hom}(H_1(\mathcal{L}M;\Z),\Z),
\] 
by interpreting a class in $H_1(\mathcal{L}M;\Z)$ as a linear combination of tori, i.e. maps $S^1\times S^1 \rightarrow M$, and $\overline{c_1}$ as integrating $c_1$ over it. Similarly, $[\omega]\in H^2(M;\reals)$ gives rise to a cohomology class 
\[
\overline{[\omega]}\in H^1(\mathcal{L}M;\reals)=\text{Hom}(H_1(\mathcal{L}M;\Z),\reals).
\]

The manifold $(M,\omega)$ is called \emph{atoroidal} if the cohomology classes $\overline{c_1}$ and $\overline{[\omega]}$ satisfy the conditions

\[
\overline{[\omega]}|_{\pi_1(\mathcal{L}M)}=0= \overline{c_1}|_{\pi_1(\mathcal{L}M)}.
\]

i.e.
\[
\displaystyle\int_{S^1\times S^1} v^*\omega=0=\displaystyle\int_{S^1\times S^1} v^*\eta \quad\quad \text{for all }v:S^1\times S^1 \rightarrow M,
\]
where $\eta$ is a $2$-form in $M$ that represents the first Chern class $c_1\in H^2(M;\Z)$.

Moreover, $(M,\omega)$ is called \emph{toroidally monotone} if 
\[
\overline{[\omega]}|_{\pi_1(\mathcal{L}M)}=\lambda_T \overline{c_1}|_{\pi_1(\mathcal{L}M)}
\]
for some non-negative constant $\lambda_{T}\in\reals$ and, if $\overline{[\omega]}|_{\pi_1(\mathcal{L}M)}$ and $\overline{c_1}|_{\pi_1(\mathcal{L}M)}$ are non-zero, the manifold is called \emph{strictly toroidally monotone}.\\

Given a cohomology class $[\theta]\in H^1(M;\reals)$, we consider the following cohomology class on the loop space $\mathcal{L}M$
\begin{eqnarray*}
\varrho:= -\overline{[\omega]} + \text{ev}^*[\theta]\in H^1(\mathcal{L}M;\reals)
\end{eqnarray*}
where $\text{ev}\colon \mathcal{L}M \rightarrow M $ denotes the evaluation map $x \mapsto x(0)$.

We say that $\varrho$ and $\overline{c_1}$ are \emph{toroidally proportional} if 

\[
\varrho|_{\pi_1(\mathcal{L}M)}=\lambda_{\varrho}\overline{c_1}|_{\pi_1(\mathcal{L}M)}
\]
i.e.

\begin{eqnarray*}\label{eqn:toroprop}
-\displaystyle\int_{S^1\times S^1} v^*\omega + \displaystyle\int_{S^1} v_0^*\theta=\lambda_{\varrho}\displaystyle\int_{S^1\times S^1} v^*\eta \quad\quad \text{for all }v:S^1\times S^1 \rightarrow M,
\end{eqnarray*}
for some non-negative $\lambda_{\varrho}\in \reals$, where $v_0\colon S^1 \rightarrow M$ is given by $v_0=v(\cdot,0)$ and $\eta$ is a $2$-form that represents the class $c_1$. 

If $\varrho|_{\pi_1(\mathcal{L}M)}$ and $\overline{c_1}|_{\pi_1(\mathcal{L}M)}$ are non-zero, we say $\varrho$ and $\overline{c_1}$ are \emph{strictly toroidally proportional}. 

Moreover, we say that $\varrho$ is \emph{rational} if the group of integrals of $\varrho$ over tori is discrete, i.e., 
\begin{eqnarray}\label{eqn:varphi_rational}
\left< \varrho, \pi_1(\mathcal{L}M)\right>=h_{\varrho}\Z
\end{eqnarray}
for some non-negative $h_{\varrho}\in \reals$. Notice that if $\varrho$ and $\overline{c_1}$ are strictly toroidally monotone  then $\varrho$ is rational and if $[\theta]=0$ then $\varrho$ and $\overline{c_1}$ are toroidally proportional if and only if $(M,\omega)$ is toroidally monotone.

Define the \emph{toroidal minimal Chern number} $c^{\min}_{1,T}\in\N$ by
\begin{eqnarray}\label{eqn:toroidalc1min}c^{\min}_{1,T}\Z=\left< \overline{c_1}, \pi_1(\mathcal{L}M)\right>,\end{eqnarray}
with convention $c^{\min}_{1,T}=\infty$ if the image of $\overline{c_1}\colon \pi_1(\mathcal{L}M) \rightarrow \Z$ is 0,
and hence, when $\varrho$ and $\overline{c_1}$ are toroidally proportional, we have
\begin{eqnarray}\label{eqn:c1minTprop}
c^{\min}_{1,T}=\frac{h_{\varrho}}{\lambda_{\varrho}}.
\end{eqnarray} Moreover define

\begin{eqnarray}\label{eqn:nut}
\nu_{T}:= \frac{c^{\min}_1}{c^{\min}_{1,T}} > 0.
\end{eqnarray} 
Notice that $c^{\min}_{1,T}\not= 0$ and, when $M$ is strictly spherically monotone, $c^{\min}_1 <\infty$. 

\begin{example}\label{example:firstchern}
	Consider $M=\CP^n\times \Sigma_{g\geq 2}$, where $\Sigma_{g\geq 2}$ is a surface with genus greater than or equal to 2. The surface $\Sigma_{g\geq 2}$ is aspherical and atoroidal and $c_{1,T}^{\min}(\CP^n)=c_{1}^{\min}(\CP^n)=n+1$. Hence, the toroidal minimal Chern number of 
	$M$ is $c_{1,T}^{\min}(M)=c_{1}^{\min}(M)=n+1$ (see, e.g., \cite{MS12}  for more details). In Theorem~\ref{maintheo}, the assumption on the minimal Chern number
	$$
	c_{1,T}^{\min}(M)\geq \dim(M)/4 +1,
	$$
	in this case, is
	$$
	n+1\geq \frac{n+1}{2} +1,
	$$
	which is equivalent to $n\geq 1$. 
\end{example}

Let $\zeta$ be a free homotopy class of maps $S^1\rightarrow M$, (i.e. let $\zeta\in \pi_0(\mathcal{L}M)$), fix a reference loop $z\in \zeta$ and a symplectic trivialization of $TM$ along $z$. Denote by $\mathcal{L}_{\zeta}M$ the component of $\mathcal{L}M$ with loops in the free homotopy class $\zeta$ (i.e. $\mathcal{L}_{\zeta}M=p_{\pi_0}^{-1}(\zeta)$ is the preimage of $\zeta$ under the natural projection $p_{\pi_0}\colon \mathcal{L}M \rightarrow \pi_0(\mathcal{L}M)$). 

Consider the connected abelian principal covering $p \colon \widetilde{\mathcal{L}}_{\zeta}M \rightarrow \mathcal{L}_{\zeta}M$ with structure group 
\begin{eqnarray}\label{eqn:Gamma_z}
\Gamma_{\zeta}:= \frac{\pi_1(\mathcal{L}_{\zeta}M)}{\ker \overline{c_1} \cap \ker \varrho}
\end{eqnarray}
where $\overline{c_1}$ and $\varrho$ are considered as homomorphisms $\pi_1(\mathcal{L}_{\zeta}M)\rightarrow \reals$ and we have $p^*\varrho=0=p^*\overline{c_1}\in H^1(\widetilde{\mathcal{L}}_{\zeta}M;\reals)$. These homomorphisms induce homomorphisms on $\Gamma_{\zeta}$:
\begin{eqnarray*} 
\overline{c_1}\colon \Gamma_{\zeta} \rightarrow \Z \quad \text{and} \quad \varrho\colon \Gamma_{\zeta} \rightarrow \reals.
\end{eqnarray*}

Define the \emph{$\zeta$-minimal Chern number} $c_{1,\zeta}^{\text{min}}\in \N$ by 
\begin{eqnarray}\label{eqn:c1minzeta}
c_{1,\zeta}^{\min}\Z= \left<\overline{c_1},H_1(\mathcal{L}_{\zeta}M;\Z)\right> \subseteq \Z
\end{eqnarray}
with the convention $c_{1,\zeta}^{\text{min}}=\infty$ if the image of $\overline{c_1}\colon H_1(\mathcal{L}_{\zeta}M;\Z)\xrightarrow{} \Z$ is $0$. The number $c_{1,\zeta}^{\text{min}}$ depends only on $(M,\omega, \zeta)$. 

\begin{remark}
	If $\zeta$ is trivial, the minimal Chern number $c_{1,\zeta=0}^{\min}$ is the usual minimal Chern number $c_{1}^{\min}$ as defined in \eqref{eqn:minimalChernnumb} but in general it is smaller. Indeed the diagram
	
	\begin{equation}\label{diagram:c1}
	\xymatrix{
		{H_1(\mathcal{L}_{\zeta}M;\Z)} \ar[r]^{\quad\quad\overline{c_1}}& {\Z}  \\
		{\pi_1(\mathcal{L}_{\zeta}M)}   \ar[u]& {\pi_2(M)} \ar[l]  \ar[u]^{c_1} .
	}
	\end{equation}
	commutes, where, if we denote by $\Omega(M,z(0))$ the space of based pointed loops with base point $z(0)$, the bottom mapping $\pi_2(M)=\pi_1(\Omega(M,z(0))\rightarrow \pi_1(\mathcal{L}_{\zeta}M)$ is induced from the mapping $\Omega(M,z(0))\rightarrow \mathcal{L}M$ given by concatenating a loop based at $z(0)$ with z. Then $c_{1,\zeta}^{\min}<\infty$ if and only if $c_{1}^{\min}<\infty$ and, in this case, $c_{1,\zeta}^{\min}$ divides $c_{1}^{\min}$ (see also Remark~\ref{rmk:nu}).
\end{remark}

We interpret a class in $H^1(\mathcal{L}_{\zeta}M;\Z)$ as a linear combination of tori, i.e., maps $v\colon S^1\times S^1 \rightarrow M$ such that the homotopy class of $v(s,\cdot)\colon S^1 \rightarrow M$ is $\zeta,$ for all $s\in S^1$. The cohomology classes $\varrho$ and $\overline{c_1}$ are said strictly $\zeta$-toroidally proportional if 
\[\varrho|_{\pi_1(\mathcal{L}_{\zeta}M)}=\lambda_{\zeta}\overline{c_1}|_{\pi_1(\mathcal{L}_{\zeta}M)},\]
for some $\lambda_{\zeta}>0$, and $\varrho|_{\pi_1(\mathcal{L}_{\zeta}M)}\not=0\not= \overline{c_1}|_{\pi_1(\mathcal{L}_{\zeta}M)}$ (see also~\cite[Section~8]{PolShe15}).
In fact, if $\varrho$ and $\overline{c_1}$ are also toroidally proportional, we have $\lambda_{\zeta}=\lambda_{\varrho}$ and moreover 
\[
\left<\varrho, \pi_1(\mathcal{L}_{\zeta}M) \right>=h_{\zeta}\Z
\]
for some non-negative $h_{\zeta}\in \reals$. Observe that, in this case,  $c^{\min}_{1,T}\not=0$ divides $c^{\min}_{1,\zeta}<\infty$ and, hence, we have 
\begin{eqnarray}\label{eqn:h_zeta}
h_{\zeta}=\frac{c_{1,\zeta}^{\min}}{c_{1,T}^{\min}}h_{\varrho},
\end{eqnarray}
with $h_{\zeta}={c_{1,\zeta}^{\min}}/{c_{1,T}^{\min}}h_{\varrho}\in \N$. 

\begin{remark}\label{rmk:nut=1}
If $\varrho$ and $\overline{c_1}$ are strictly toroidally proportional and $\zeta$-toroidally proportional, then, when $\nu_T=1$ (recall \eqref{eqn:nut}), $c_{1}^{\min}=c_{1,T}^{\min}=c_{1,\zeta}^{\min}$.
\end{remark}

\subsubsection{The covering space $\widetilde{\mathcal{L}}_{\zeta}M$.}\label{section:decrp_cov_sp}
	Consider the space of pairs $(x,v)$ where $x\colon S^1 \rightarrow M$ is a loop in $M$ with homotopy class $\zeta$ and $v$ a cylinder (i.e. $v\colon [0,1]\times S^1 \rightarrow M$ is a homotopy) connecting $x$ to $z$. The cylinder $v$ is called a capping of $x$. 
	\begin{remark}
		Observe that if $\zeta$ is the trivial homotopy class and $z$ is a constant loop then $v$ can be viewed as a disc in $M$.
	\end{remark}
	We say that $(x,v)$ is in relation with $(x',v')$ if the following conditions hold:
\[x=x{'},\quad\quad
\displaystyle\int_{S^1\times S^1} u^*\eta =0
\]
and
\[
	\int_{S^1\times S^1} u^*\omega = \int_{S^1} {u_0}^*\theta
	\]
	where $u$ is the two-torus obtained by attaching $v$ and $-v'$ (that denotes $v'$ with reversed orientation) to each other, $\eta$ is a $2$-form representing the first Chern class $c_1$ and $u_0=u|_{S^1 \times \{0\}}$ (see Figure~\ref{figure:u0}). The equivalence class $\overline{x}=[x,v]$ is called a lift of the loop $x\in\mathcal{L}_{\zeta}M$ and the space of such equivalence classes is denoted by $\widetilde{\mathcal{L}}_{\zeta}M$.
	
\begin{figure}[htb!]
	\centering
	\def\svgwidth{150pt}
	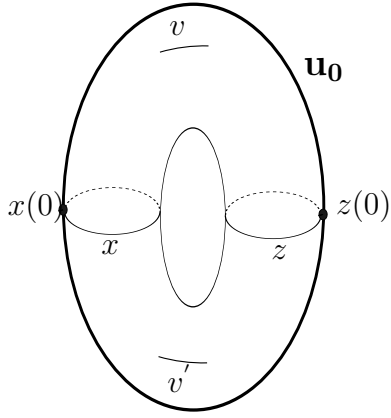
	\caption{The circle $u_0$.}\label{figure:u0}
\end{figure}
	
	\subsection{Periodic orbits and the action functional}\label{section:po_af} In this section, we follow \cite{BH01_non-con},  \cite{LO_fixedpts95} and references there in.

	Let $\phi\in \text{Symp}_0(M)$, $\phi_t$ be a symplectic isotopy connecting $\phi_0=id$ to $\phi_1=\phi$ with $\text{Flux}([\phi_t])=:[\theta]$ and $X_t$ be the vector field associated with $\phi_t$, i.e. 
	\[
	\frac{d}{dt}\phi_t = X_t \circ \phi_t.
	\]
	Recall that $\text{Flux}\colon \widetilde{\text{Symp}_0}(M,\omega) \rightarrow H^1(M,\reals)$ is the homomorphism defined by
	\begin{eqnarray}\label{eqn:flux}
	[\phi_t] \mapsto \left[ \int_0^1 \omega (X_t, \cdot) dt\right]
	\end{eqnarray}
	and its kernel is given by $\widetilde{\text{Ham}}(M,\omega)$, i.e. the universal covering of the group of Hamiltonian diffeomorphisms.

L\^e and Ono proved in \cite[Lemma~2.1]{LO_fixedpts95} that we can deform $\{\phi_t\}$ through symplectic isotopies (keeping the end points fixed) so that the cohomology classes $[\omega(X^{\prime}_t,\cdot)]$, for all $t\in[0,1]$, and $\text{Flux}([\phi'_t])=[\theta]$ are the same (where $X^{\prime}_t$ is the vector field associated with the deformed symplectic isotopies $\phi'_t$). Namely, each element in $\widetilde{\text{Symp}_0}(M,\omega)$ admits a representative symplectic isotopy generated by a smooth path of closed $1$-forms $\theta_t$ on $M$ whose cohomology class is identically equal to the flux, i.e.   
\begin{eqnarray}\label{eqn:Ham}
-\omega(X^{\prime}_t,\cdot)= \theta + dH_t=:\theta_t
\end{eqnarray}
for some one-periodic in time Hamiltonian $H_t\colon M \rightarrow M, \;t\in S^1$. 

The fixed points of $\phi=\phi_1$ are in one-to-one correspondence with one-periodic solutions of the differential equation
\begin{eqnarray}\label{eqn:de_theta}
\dot{x}(t)=X_{\theta_t}(t,x(t))
\end{eqnarray}
where $X_{\theta_t}$ is defined by $\omega(X_{\theta_t},\cdot)=-\theta_t$. From now on we denote the vector field $X_{\theta_t}$ by $Z_t$.

The set of one-periodic solutions of \eqref{eqn:de_theta} is denoted by $\mathcal{P}(\theta_t)$ and coincides with the zero set of the closed $1$-form $\alpha:=\alpha_{\{\phi_t\}}$ defined on the loop space of $M$, $\mathcal{L}M$, by
\begin{eqnarray*}\label{eqn:mv-af}
\alpha_{\{\phi_t\}}(x,\xi)&=& \displaystyle\int_0^1 \omega(\dot{x}(t)-Z_t(x(t)), \xi) dt\\
&=&\displaystyle\int_0^1 \omega(\dot{x}(t), \xi) + \theta_t(x(t))(\xi) dt\\
&=& \displaystyle\int_0^1 \omega(\dot{x}(t), \xi) dt + \displaystyle\int_0^1 (\theta+ dH_t)(x(t))(\xi) dt
\end{eqnarray*}
where $x\in \mathcal{L}M$ and $\xi\in T_x\mathcal{L}M$ (i.e. $\xi$ is a tangent vector field along the loop $x$).	
	
From \eqref{eqn:Gamma_z} we then see that $p^*\alpha=d\CA$ for some $\CA\in C^{\infty}(\widetilde{\mathcal{L}}_{\zeta}M,\reals)$.
	
	Then an action functional $\CA=\mathcal{A}_{\{\phi_t\}}$ is given by 
	\begin{eqnarray}\label{eqn:action_flnov}
	\mathcal{A}_{\{\phi_t\}}([x,v])=-\int_{[0,1]\times S^1} v^*\omega +\displaystyle\int_0^1 \left[\left(\int_0^1 (v_t)^*\theta\right)+ H_t(x(t)) \right]dt
	\end{eqnarray}
	where $[x,v]\in \widetilde{\mathcal{L}}_{\zeta}M$ is as in Section~\ref{section:decrp_cov_sp} and $v_t\colon [0,1]\rightarrow M$ is defined by $s \mapsto v(s,t)$ (see Figure~\ref{figure:vt}).
	
		\begin{figure}[htb!]
			\centering
			\def\svgwidth{100pt}
			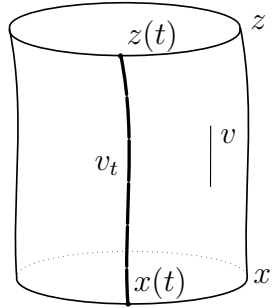
			\caption{The segment $v_t$.}\label{figure:vt}
		\end{figure}
		
	The action functional $\mathcal{A}_{\{\phi_t\}}$ is homogeneous with respect to iterations in the following sense
	\[
	\mathcal{A}_{\{\phi_t^{ k}\}}([x,v]^{k})=k\mathcal{A}_{\{\phi_t\}}([x,v]),
	\]
	where $[x,v]^{k}$ is the $k$-th iteration of $[x,v]$, and satisfies
	\[	\CA_{\{\phi_t\}}([x,v]\#A)=\CA_{\{\phi_t\}}([x,v]) + \varrho (A),
	\]
	where $A\in \pi_1(\mathcal{L}_{\zeta}M, z)$.
	
	\begin{remark}\label{rmk:action}
		\begin{enumerate}
			\item\label{rmk:action1} The action of a (lifted) periodic orbit depends on the initial choice of $z$. If $z'\in \zeta$, there is a homotopy $F\colon [0,1]\times S^1 \rightarrow M$ connecting $z$ to $z'$, i.e., $F(0, \cdot)=z$ and $F(1, \cdot)=z'$. Then $\CA_{\{\phi_t\}}([x,v]\#F)=\CA_{\{\phi_t\}}([x,v])+\CA(F)$ where $[x,v]\#F$ denotes the class $[x,w]$ with $w$ the cylinder obtained by concatenating $v$ with $F$ and		
			 \[\CA(F)=-\displaystyle\int_{[0,1]\times S^1} F^*\omega + \displaystyle\int_{0}^1 (\displaystyle\int_{0}^1 F_t^*\theta) dt\] with $F_t\colon [0,1] \rightarrow M$, $F_t(s):=F(t,s), t\in S^1.$ 
			\item Whenever we consider the $k$-th iteration $\phi_t^k$ of $\phi_t$ we simultaneously iterate the class $\zeta$ and the reference curve $z$ (i.e. we pass to $k\zeta$ and $z^k$, respectively).
		\end{enumerate}
	\end{remark}

	\subsection{The mean index and augmented action}	Let $x$ be a periodic orbit of $Z_t$. If the eigenvalues of the map
	\[
	d\phi_{x(0)}:T_{x(0)}M\rightarrow T_{x(0)}M
	\]
	are not equal to one, then the orbit $x$ is called \emph{non-degenerate}. If, in addition, none of the eigenvalues of the linearized return map $d\phi_{x(0)}$ has absolute value equal to one, then we say $x$ is \emph{hyperbolic}. The symplectomorphism $\phi$ is said to be non-degenerate if all its one-periodic orbits are non-degenerate. Moreover, if all periodic orbits of $Z_t$ with homotopy class $\zeta$ are non-degenerate, then the set of periodic orbits of $Z_t$ in $\mathcal{L}_{\zeta}M$, denoted by $\mathcal{P}_{\zeta}$, is finite.
	
	The homotopy class $\zeta$ gives rise to a well defined, up to homotopy, $\C$-vector bundle trivialization of $x^*TM$ for every $x\in \mathcal{L}_{\zeta}M$ and, for a one-periodic orbit of $\phi$, the linearized flow along $x$
	\[
	d\phi_t\colon T_{x(0)}M \rightarrow T_{x(t)}M
	\]
	can be viewed as a symplectic path $\Phi\colon[0,1] \rightarrow \Sp(2n).$ Set $\overline{\mathcal{P}_{\zeta}}:= p^{-1}(\mathcal{P}_{\zeta})$. Then one gets a well defined mean index 
	$\D_{\phi_t}\colon \overline{\mathcal{P}_{\zeta}} \rightarrow \reals$ and, for non-degenerate orbits, the Conley--Zehnder index $\MUCZ([x,v])\in \Z$ satisfying
	\[
	\D_{{\{\phi_t^{k}\}}}([x,v]^{k})=k\D_{\{\phi_t\}}([\widetilde{x},v]),
	\]
	\[
	\D_{\{\phi_t\}}([x,v]\#A)=\D_{\{\phi_t\}}([x,v])- 2\overline{c_1}(A)
	\]
	and
	\[
	\MUCZ([x,v]\#A)=\MUCZ([x,v])-2\overline{c_1}(A) \quad\quad\text{(when } x\text{ is non-degenerate),}
	\]
	for every $[x,v]\in \overline{\mathcal{P}_{\zeta}}$ and $A\in \Gamma_{\zeta}$ (cf. \cite{BH01_non-con} and \cite{SZ92}). Moreover
	\begin{eqnarray}\label{eqn:mi_czi}
	0\not =|\D_{\{\phi_t\}}([x,v])-\MUCZ([x,v])|< n,
	\end{eqnarray}
	when $x$ is non-degenerate.
	The mean index $\D_{\{\phi_t\}} ([x,v])\in \reals$ measures, roughly speaking, the total angle swept by certain eigenvalues with absolute value one of the symplectic path $\Phi$ induced by the linearized flow and a trivialization as above.
	
	\begin{remark}\label{rmk:index}
		The Conley--Zehnder index depends on the initial choice of the trivialization of $TM$ along $z$; see \cite[Section~3]{BH01_non-con}.
	\end{remark}

	The \emph{augmented action} is defined by
	\begin{eqnarray}\label{eqn:augmentedaction}
	\underline{\mathcal{A}}_{\{\phi_t\}}([x,v]):= \mathcal{A}_{\{\phi_t\}}([x,v]) - \frac{\lambda_{\varrho}}{2} \D_{\{\phi_t\}}([x,v])
	\end{eqnarray}
	and it is homogeneous with respect to iterations in the following sense
	\[
	\underline{\mathcal{A}}_{\{\phi_t^{k}\}}([x,v]^{k})=k\underline{\mathcal{A}}_{\{\phi_t\}}([x,v]).
	\]
	\begin{remark}\label{rmk:augmented_independent}
		Note that the augmented action is independent of the lift of a loop, i.e.
		\[
		\underline{\mathcal{A}}_{\{\phi_t\}}([x,v])=\underline{\mathcal{A}}_{\{\phi_t\}}([x,v{'}])
		\]
		where $v$ and $v{'}$ are cylinders connecting $x$ to $z$. 
	\end{remark}
	
	\subsection{The Floer--Novikov homology}\label{section:fnhomol} 
	In this section, we revisit the definition of the Floer--Novikov homology for non-contractible periodic orbits following \cite{BH01_non-con}. Consider a smooth almost complex structure $J$ on $M$ compatible with $\omega$, i.e. such that 
	\[
	g(X,Y):=\omega(X,JY)
	\]
	defines a Riemannian metric on $M$. We will denote by $\mathcal{J}$ the set of almost complex structures compatible with $\omega$.
	Choose $J\in \mathcal{J}$ and let $\widetilde{g}$ denote the induced weak Riemannian metric on $\mathcal{L}M$ given by
	\[
	\widetilde{g}(X_x,Y_x)=\displaystyle\int_{S^1} g(X_x(t), Y_x(t)) dt,
	\]
	where $X_x$ and $Y_x$ are vector fields along $x$. A gradient flow line is a mapping $u\colon \reals\times S^1\rightarrow M$ satisfying
	\begin{eqnarray}\label{eqn:gradientflowline}
	\partial_s u + J(\partial_t u - Z_t(u(s,t)))=0.
	\end{eqnarray}
	The maps $u\colon \reals \rightarrow \mathcal{L}M$ which satisfy \eqref{eqn:gradientflowline} with boundary conditions
	\begin{eqnarray}\label{eqn:boundconds_lim_tilde}
	\displaystyle\lim_{s\rightarrow\pm\infty} u(s,t) =  x_{\pm}(t)\in \mathcal{P}_{\zeta}
	\end{eqnarray}
	can be seen as connecting orbits between $\overline{x}_{-}:= [x_{-}, v_{-}]$ and $\overline{x}_{+}:= [{x}_{+}, v_{+}]$ in $\widetilde{\mathcal{L}}_{\zeta}M$ where 
	\begin{eqnarray}\label{eqn:transform_cap}
	v^{-}\#u=v^{+}.
	\end{eqnarray} 
	We denote by $\CM(\overline{x}_{-},\; \overline{x}_{+})$ the space of finite energy solutions of (\ref{eqn:gradientflowline}--\ref{eqn:boundconds_lim_tilde}) that transform \emph{cappings} as in \eqref{eqn:transform_cap}. The energy of a connecting orbit in this space is given by  
	\[
	E(u):= \displaystyle\int_{\reals\times S^1} |\p_s u|^2_{g} dsdt=\CA_{\{\phi_t\}}(\overline{x}_{+})-\CA_{\{\phi_t\}}(\overline{x}_{-}).
	\]
	
	For any $\overline{x}_{-},\;\overline{x}_{+}\in \overline{\mathcal{P}_{\zeta}},$ the space $\CM(\overline{x}_{-},\; \overline{x}_{+})$ is a smooth manifold of dimension $\MUCZ(\overline{x}_{+})-\MUCZ(\overline{x}_{-})$. It admits a natural $\reals$-action given by reparametrization. For $\overline{x}_{-},\; \overline{x}_{+} \in \overline{\mathcal{P}_{\zeta}}$ such that $\MUCZ(\overline{x}_{+})-\MUCZ(\overline{x}_{-})=1$, we have that $\CM(\overline{x}_{-},\; \overline{x}_{+})/ \reals$ is finite and set
	\[
	n_2(\overline{x},\; \overline{y}):= \# \CM(\overline{x},\; \overline{y}) \quad \text{modulo two}.
	\]
	
	Denote by $\overline{\mathcal{P}_{\zeta,k}}$ the subset of $\overline{\mathcal{P}_{\zeta}}$ whose elements $\overline{x}$ satisfy
	${\MUCZ(\overline{x})=k}$. Consider the chain complex where the $k$-th chain group $C_k$ consists of all formal sums
	\begin{eqnarray*}
		\sum \xi_{\overline{x}}\cdot \overline{x}
	\end{eqnarray*}
	with $\overline{x}\in \overline{\mathcal{P}_{\zeta,k}},\;\xi_{\overline{x}}\in \mathbb{Z}_2$ and such that, for all $c\in \reals$, the set
	\[
	\big{\{}\overline{x}\;|\;\xi_{\overline{x}}\not=0,
	\;\mathcal{A}_{\{\phi_t\}}(\overline{x})>c\big{\}}
	\]
	is finite. For a generator $\overline{x}$ in $C_k$, the boundary operator $\partial_k$ is defined as follows
	\[
	\partial_k(\overline{x})=\displaystyle\sum_{\MUCZ(\overline{y})=k-1} n_2(\overline{x},\overline{y})\overline{y}.
	\]
	The boundary operator $\partial$ satisfies $\partial^2=0$ and we have the Floer--Novikov homology groups
	\begin{eqnarray*}
	HFN_k([\theta], \zeta)=\frac{\ker\partial_k}{\im \partial_{k+1}}.
	\end{eqnarray*}
	
	\begin{remark}[\cite{GG:hyp12}]
The pair-of-pants product is not defined on the "non-contractible" Floer--Novikov homology for an individual class $\zeta$. Such a product between $HF_*([\theta], \zeta_1)$ and $HF_*([\theta], \zeta_2)$ takes values in $HF_*([\theta], \zeta_1+\zeta_2)$.
	\end{remark}
	
	\subsubsection{The Filtered Floer--Novikov homology}
	The total chain Floer complex $C_*=\colon$ $C_*^{(-\infty, \infty)}$ admits a filtration by $\reals$. Let $\CS$ be the set of critical values of the functional $\mathcal{A}_{\{\phi_t\}}$ (defined in \eqref{eqn:action_flnov}) which is called the \emph{action spectrum} of $\{\phi_t\}$. For each
	$b\in(-\infty,\infty]$ outside $\CS$, the chain complex $C_*^{(-\infty, b)}$ is generated by lifted loops $\overline{x}$ with action $\A_{\{\phi_t\}}(\overline{x})$ less than~$b.$ For $-\infty\leq a < b \leq \infty$ outside $\CS,$ set
	\[
	C_*^{(a,b)} :=C_*^{(-\infty,b)} / C_*^{(-\infty,a)}.
	\]
	The boundary operator $\p\colon C_*\rightarrow C_{*-1}$ descends to $C_*^{(a,b)}$ and hence the \emph{filtered Floer--Novikov homology} $HFN_*^{(a,b)}(\theta_t, \zeta)$ is well defined.
	
	This construction also extends to all symplectomorphisms in $\text{Symp}_0(M,\omega)$ (see \cite[Sections~3~and~4]{BH01_non-con}). For a path $\{\phi_t\}$ connecting an arbitrary $\phi\in\text{Symp}_0(M,\omega)$ to the identity, set
	\begin{eqnarray}\label{eqn:homol_deg}
	HFN_*^{(a,b)}(\theta_t, \zeta)\colon=
	HFN_*^{(a,b)}(\theta_t^{\prime}, \zeta)
	\end{eqnarray}
	where $\theta_t$ is the $1$-form associated with the path $\{\phi_t\}$ and $\theta_t^{\prime}$ is the $1$-form associated with a non-degenerate perturbation $\{\phi'_t\}$ of $\{\phi_t\}$. Here $-\infty\leq a<b\leq \infty$ are outside the action spectrum of $\{\phi_t\}$.

	\begin{remark}
		Considering Remark~\ref{rmk:action}~\ref{rmk:action1} and Remark~\ref{rmk:index}, observe that the filtered Floer--Novikov homology when $\zeta=0$ is a shift of the standard grading and filtration in the (contractible) Floer--Novikov homology (considered in \cite{O_flux06}).
	\end{remark}
	
	\subsection{Novikov rings and quantum homology}
	In this section we introduce the Novikov rings used in this paper in order to define the quantum homology and the quantum product action. Here we follow \cite{GG:hyp12} and references therein.  
	\subsubsection{The Novikov ring $\Lambda$ and quantum homology with coefficients in $\Lambda$.}\label{section:novikov_L}	Consider the Novikov ring $\Lambda:=\Lambda(\Gamma, [\omega], \F)$ associated with $\Gamma$ and the weighting homomorphism $[\omega]\colon \Gamma \rightarrow \F$ with values in the field $\F=\Z_2$ (see \cite{HS95}). Here the group $\Gamma$ is the quotient
	\[
	\Gamma:= \frac{\pi_2(M)}{\ker c_1 \cap \ker [\omega]}
	\]
	where $[\omega]\colon \pi_2(M)\rightarrow \reals$ and $c_1\colon \pi_2(M)\rightarrow \Z$ are the maps given by the integration of the corresponding $2$-forms. These maps descend to $\Gamma$. When $M$ is strictly monotone, the Novikov 
	ring $\Lambda$ may simply be taken to be the group algebra $\F[\Gamma]$ over $\F$.
	Moreover, in this case, the group $\Gamma$ is isomorphic to $\Z$ and denote by $A_0$ the generator of $\Gamma$ with $c_1(A_0)=-2c_1^{\text{min}}$. 
	Furthermore, an element in $\Lambda=\F[\Gamma]$ is a formal finite linear combination of $\alpha_A e^{A}$, with $\alpha_A\in \F$ and $A\in \Gamma$, and $\Lambda$ is graded by setting $\deg(e^{A})=-2c_1(A)$ with $A\in \Gamma$. Hence, setting $q=e^{A_0}$, we have $\deg(q)=-2c_1^{\text{min}}$ and $\Lambda=\F[q,q^{-1}]$. \\
	
	The \emph{quantum homology} of $M$ is defined 
	as 
	\[
	HQ_*(M)= H_*(M)\otimes\Lambda
	\]
	where the degree of a generator $\alpha \otimes e^{A}$, with $\alpha\in H_*(M)$ and $A\in \Gamma$, is $\deg(\alpha)+2c_1(A)$. The product structure is given by the quantum product:
	\[
	\alpha*\beta =\displaystyle\sum_{A\in \Gamma} (\alpha*\beta)_A e^A
	\]  
	where $(\alpha*\beta)_{A}\in H_*(M)$ is defined via some Gromov--Witten invariants of $M$ and has degree ${\deg(\alpha)+\deg(\beta)-2n+2c_1(A)}$. Thus \[{\deg(\alpha*\beta)=\deg(\alpha)+\deg(\beta)-2n}.\] When $A=0$, $(\alpha*\beta)_0=\alpha\cap\beta$, where $\cap$ stands for the intersection product of ordinary homology classes.
	
	It suffices to restrict the summation to the negative cone $[\omega](A)\leq 0$ and we can write
	\[
	\alpha*\beta=\alpha\cap\beta+\displaystyle\sum_{k>0}(\alpha*\beta)_k q^k,
	\]
	where $\deg((\alpha*\beta)_k)=\deg(\alpha)+\deg(\beta)-2n+2c^{\min}_{1}k$ and the sum is finite.
	
	The unit in the algebra $HQ_*(M)$ is the fundamental class $[M]$ and, for $a\in \Lambda$ and $\alpha\in H_*(M)$, 
	\[a\alpha=(a[M])*\alpha\] where degree of $a\alpha$ is $\deg(a\alpha)=\deg(a)+\deg(\alpha)$. Then the ordinary homology $H_*(M)$ is canonically embedded in $HQ_*(M)$.

	The map $[\omega]\colon \Gamma \rightarrow \reals$ can be extended to $HQ_*(M)$ by
	\[
	[\omega](\alpha)=\max\{[\omega](A)\;|\; \alpha_{A} \not= 0 \}=\max\{-h_0k\;|\;\alpha_k\not=0\}
	\]
	where $\alpha=\sum\alpha_{A} e^A=\sum\alpha_k q^k$ and it satisfies
	\begin{eqnarray*}
		[\omega](\alpha+\beta)\leq\max\{[\omega](\alpha),[\omega](\beta)\}
	\end{eqnarray*}
	and
	\begin{eqnarray*}\label{eqn:I_omega_inequality_2}
	[\omega](\beta*\alpha)\leq [\omega](\alpha)+[\omega](\beta).
	\end{eqnarray*}
	
	\begin{example}\label{example:quantumprod}
		Consider $M=\CP^n,\;\alpha=[\CP^{n-1}]\in H_{2n-2}(M)$ and $\beta=[\text{pt}]\in H_{0}(M)$ (where $[\text{pt}]$ is the class of a point in $\CP^n$). The fact that there is a unique line through any two points is reflected in the identity $(\beta*\alpha)_A=[\CP^n]$ where $c_1(A)=c_1^{\min}$. Hence, in $HQ_*(\CP^n)$, $\beta*\alpha=q[\CP^n].$ 
		Similarly, for    
		$M= \CP^n\times \Sigma_{g\geq 2}$,
		\[
		\alpha=[\CP^{n-1}\times \Sigma_{g\geq 2}]\in H_{2n}(M)
		\] 
		and 
		\[
		\beta=[\text{pt} \times \Sigma_{g\geq 2}]\in H_{2}(M),
		\] 
		the quantum product $\beta*\alpha$ satisfies the homological condition \eqref{l:1}. (See \cite{MS12} for these and further computations.)
	\end{example} 
	
\begin{remark}\label{rmk:exampleManifolds}
	The manifold $\CP^n\times \Sigma_{g\geq 2}$ admits symplectomorphisms which need not be Hamiltonian diffeomorphisms and, as seen in Examples~\ref{example:firstchern} and ~\ref{example:quantumprod}, satisfies the conditions on the manifold of Theorem~\ref{maintheo}.
	
	Observe that $\CP^n$ also satisfies these conditions. However, $\CP^n$ is simply connected and hence (see~\eqref{eqn:flux}) has no symplectomorphism which is not a Hamiltonian diffeomorphism.
	
	Let us also point out that there are symplectic flows on $\Sigma_{g\geq 2}$ with only (finitely many) hyperbolic fixed points and no other periodic orbits (see e.g.~\cite[Exercise~14.6.1, Chapter~14]{KH95}). 
\end{remark}

	\subsubsection{\emph{The Novikov ring $\Lambda_{\zeta}$.}}\label{section:novikov_Lz} For a given homotopy class $\zeta$, consider the Novikov ring $\Lambda_{\zeta}:=\Lambda(\Gamma_{\zeta}, \varrho, \F)$ associated with $\Gamma_{\zeta}$ (defined in \eqref{eqn:Gamma_z}) and weighting homomorphism $\varrho\colon \Gamma_{\zeta} \rightarrow \reals$ (defined in \eqref{eqn:toroprop}) with values in the field $\F=\Z_2$. When the homomorphisms $\varrho$ and $\overline{c_1}\colon H_1(M;\Z)\rightarrow \reals$ are strictly toroidally proportional, the Novikov ring $\Lambda_{\zeta}$ may be taken to be the group algebra $\F[\Gamma_{\zeta}]$ over $\F$. For all $\zeta$,  
	\[
	\Lambda_{\zeta}=\F[q_{\zeta}, q_{\zeta}^{-1}]
	\]
	where $q_{\zeta}=e^{A_{\zeta}}$ and $A_{\zeta}$ is the element which generates $\Gamma_{\zeta}\cong \Z$ with degree $-2c^{\min}_{1,\zeta}$. This isomorphism holds under the hypotheses of the main theorem, namely, that $\varrho$ and $\overline{c_1}$ are $\zeta$-strictly toroidally proportional.

	\begin{remark}\label{rmk:nu}
		When $M$ is strictly spherically monotone, we have $c_{1}^{\text{min}}<\infty$. In this case, by diagram~\eqref{diagram:c1}, we have ${\text Im}(\pi_2(M)\xrightarrow{c_1}\Z) \subset {\text Im}(H_1(\mathcal{L}_{\zeta}M;\Z)\xrightarrow{\overline{c_1}}\Z) $. Hence $c_{1,\zeta}^{\text{min}}$ divides $c_{1}^{\text{min}}$ and $q=q^{\nu_{\zeta}}_{\zeta}$ with $\nu_{\zeta}=c_{1}^{\text{min}} / c_{1,\zeta}^{\text{min}}\in \N$. Moreover, under the conditions of Remark~\ref{rmk:nut=1} we have $q=q_{\zeta}$.
	\end{remark}
	
	\subsubsection{\emph{The quantum product action.}}\label{section:capproduct} 
	We describe an action of the quantum homology $HQ_*(M)$ on the filtered Floer--Novikov homology. We follow  \cite[Section~2.3]{GG:hyp12} in the Floer--Novikov setting; see \cite[Section~3]{LO_cup96} for more details. 
	
	Let $\phi$ be a non-degenerate symplectomorphism, $J$ be a generic almost complex structure and $[\sigma]$ be a class in $H_*(M)$. Denote by $\mathcal{M}(\overline{x},\overline{y};\sigma)$ the moduli space of solutions $u$ of \eqref{eqn:gradientflowline}--\eqref{eqn:boundconds_lim_tilde} that transform cappings as in \eqref{eqn:transform_cap} with $\overline{x}, \overline{y}\in \overline{\mathcal{P}}_{\zeta}$ and such that $u(0,0)\in \sigma$ where $\sigma$ is a generic cycle representing $[\sigma]$.
	
	Then the dimension of this moduli space is
	\[
	\dim \;\mathcal{M}(\overline{x},\overline{y};\sigma)=\MUCZ(\overline{x})-
	\MUCZ(\overline{y})
	-\text{codim}(\sigma).
	\]
	
	Let $m(\overline{x},\overline{y};\sigma)\in\Z_2$ be $\# \mathcal{M}(\overline{x},\overline{y};\sigma)$ modulo 2 when this moduli space is zero-dimensional and zero otherwise. For any $c,\;c'\not\in \mathcal{S}$, there is a map
	\[
	\Phi_{\sigma}:C_*^{(c,c')}\rightarrow C_{*-\text{codim}(\sigma)}^{(c,c')}
	\]
	induced by
	\[
	\Phi_{\sigma}(\overline{x})=\displaystyle\sum_{\overline{y}} m(\overline{x},\overline{y};\sigma)\overline{y}.
	\]
	This map commutes with the Floer--Novikov differential $\p$ and descends (independently of the choice of the cycle representing the class $[\sigma]$) to a map
	\[
	\Phi_{[\sigma]}:HFN_*^{(c,c')}(\theta_t, \zeta)\rightarrow HFN_{*-\text{codim}(\sigma)}^{(c,c')}(\theta_t, \zeta).
	\]
	The action of a class $\alpha=q^{k}[\sigma]$ is induced by the map 
	\begin{eqnarray}\label{eqn:action_induced}
	\Phi_{q^{k}\sigma}(\overline{x}):=\displaystyle\sum_{\overline{y}}m(q^{k}\overline{x}, \overline{y}; \sigma)\overline{y}
	\end{eqnarray}
	
	Recall that, by Remark~\ref{rmk:nu}, $q=q_{\zeta}^{\nu_{\zeta}}$ with $\nu_{\zeta}\in \N.$ Then, in \eqref{eqn:action_induced}, the element $q^{k}\overline{x}$, with $\overline{x}=[x,v]$, is the lift $[x,w]$ of $x$ where $w$ is obtained by attaching $k\nu_{\zeta} A_{\zeta}$ to $v$.
	
	The map $\Phi_{\alpha}$ can be extended to all $\alpha\in H_*(M)\otimes\Lambda$ by linearity over $\Lambda$ and then we have 
	\begin{eqnarray*}\label{eqn:action_quantum}
	\Phi_{\alpha}: HFN_*^{(c,c')}(\theta_t, \zeta)\rightarrow HFN_{*-2n+\deg(\alpha)}^{(c,c')+\varrho(\alpha)}(\theta_t,\zeta).
	\end{eqnarray*}
	
	\begin{remark}\label{rmk:varphi_zeta} Note that the map $\varrho\colon \Gamma_{\zeta} \rightarrow \reals$ extends to $ H_{*}(M)\otimes \Lambda $ as
		\[
		\varrho(\alpha):=\max\{-k\nu_{\zeta}h_{\zeta}\;|\;\alpha_k\not=0\}
		=\max\{-k\nu_{T}h_{\varrho}\;|\;\alpha_k\not=0\}
		\]
		
		where $\alpha=\sum\alpha_k q^k=\sum\alpha_k q_{\zeta}^{k\nu_{\zeta}}$. The above equality follows from 
		\begin{eqnarray}\label{eqn:nus}
		\nu_{\zeta}h_{\zeta} = \nu_{\zeta} \frac{c_{1,\zeta}^{\min}}{c_{1,T}^{\min}}h_{\varrho}=\frac{c_{1}^{\min}}{c_{1,T}^{\min}}h_{\varrho}=\nu_{T}h_{\varrho} 
		\end{eqnarray}
		which holds by \eqref{eqn:h_zeta}, Remark~\ref{rmk:nu} and \eqref{eqn:nut}.
		
		  We have
		\begin{eqnarray*}
			\varrho(\alpha+\beta)\leq\max\{\varrho(\alpha), \varrho(\beta)\}
		\end{eqnarray*}
		and
		\begin{eqnarray}\label{eqn:varphi_inequality_2}
		\varrho(\alpha*\beta)\leq \varrho(\alpha)+\varrho(\beta). 
		\end{eqnarray}
	\end{remark}
	
	The maps $\Phi_{\alpha}$, for all $\alpha\in HQ_*(M)$, give an action of the quantum homology on the filtered Floer--Novikov homology.\\
	
	This action has the following properties:
	\[
	\Phi_{[M]}=id
	\]
	and
	\begin{eqnarray}\label{eqn:multip_quantum}
	\Phi_{\beta}\Phi_{\alpha}=\Phi_{\beta*\alpha}.
	\end{eqnarray}
	
	\begin{remark}
		Observe that in the multiplicativity property \eqref{eqn:multip_quantum}, in general, the maps on both side of the equality have different target spaces. We should understand the identity in \eqref{eqn:multip_quantum} as that the following diagram commutes for any interval $(d,d^{'})$ with  $d\geq c+\varrho(\alpha)+\varrho(\beta)$ and $d^{'}\geq c'+\varrho(\alpha)+\varrho(\beta)$:
		\begin{equation}\label{diagram:cap}
		\xymatrix{
			{HFN_*^{(c,c^{\prime})}(\theta_t,\zeta)} \ar[r]^{\Phi_{\alpha}}\ar[dr]_{\Phi_{\beta*\alpha}}
			& {HFN_{*-2n+\deg(\alpha)}^{(c,c^{\prime})+\varrho(\alpha)}(\theta_t,\zeta)} \ar[r]^{\Phi_{\beta}} & {HFN_{*-4n+\deg(\alpha)+\deg(\beta)}^{(c,c^{\prime})+\varrho(\alpha)+\varrho(\beta)}}(\theta_t,\zeta) \ar[d] \\
			{}                     & {HFN_{*-2n+\deg(\beta*\alpha)}^{(c,c^{\prime})+\varrho(\beta*\alpha)}(\theta_t,\zeta)}             \ar[r]   &  {HFN_{*-2n+\deg(\beta*\alpha)}^{(d,d^{'})}(\theta_t,\zeta)}
		}
		\end{equation} where the vertical and the bottom horizontal arrows are the natural quotient-inclusion maps whose existence is guaranteed by the choice of $d$ and $d^{'}$ and \eqref{eqn:varphi_inequality_2}.
	\end{remark}
	
	\section{Proof of the main theorem}\label{section:proof}
	
	\subsection{The Ball-Crossing Energy Theorem}\label{section:BCEthm} In this section, we state the theorem which corroborates the importance of the assumption on the hyperbolicity of the orbit $\gamma$. See \cite{GG:hyp12} for the original statement and proof of the theorem and a discussion on the result.   
	
	Let $\phi$ be a symplectomorphism (isotopic to the identity) on a symplectic manifold $(M,\omega)$ and fix a one-periodic in time almost complex structure $J$ compatible with $\omega$. For a closed domain $\Sigma \subset \reals \times S^1_k$ (i.e. a closed subset with non-empty interior), where $S^1_k=\reals/k\Z$, the \emph{energy} of a solution ${u:\Sigma \rightarrow  M}$ of the equation \eqref{eqn:gradientflowline} is, by definition,
	\[
	E(u):= \displaystyle\int_{\Sigma} |\p_s u|^2_{g} dsdt.
	\]
	
	Let $\gamma$ be a hyperbolic one-periodic solution of $Z_t$ in $M$ and $\overline{\gamma}\in \widetilde{\mathcal{L}}_{\zeta}M$ be a lift of $\gamma$. 
	
	A solution ${u:\Sigma \rightarrow M}$ of the equation \eqref{eqn:gradientflowline} is said to be \emph{asymptotic to} $\gamma^k$ as $s\rightarrow \infty$ if $\Sigma$ contains a cylinder $[s_0,\infty)\times S^1_k$ and $u(s,t)\rightarrow \gamma^k(t)$ $C^{\infty}$-uniformly in $t$ as $s\rightarrow \infty$.
	
	Consider a small closed neighborhood U of $\gamma$ with smooth boundary.
	\begin{theorem}[({\cite[Ball-Crossing Energy Theorem]{GG:hyp12}})]
		\label{thm:ballcrossing}
		\
		
		There exists a constant $c_{\infty}>0$ (independent of $k$ and $\Sigma$) such that for any solution $u$ of the equation \eqref{eqn:gradientflowline}, with $u(\partial\Sigma)\subset\partial U$ and $\partial\Sigma\not=\emptyset$, which is asymptotic to $\gamma^k$ as $s\rightarrow\infty$, we have
		\begin{eqnarray}\label{eqn:energy_bound}
		E(u)>c_{\infty}.
		\end{eqnarray}
		Moreover, the constant $c_{\infty}$ can be chosen to make \eqref{eqn:energy_bound} hold for all $k$-periodic almost complex structures (with varying $k$) $C^{\infty}$-close to $J$ uniformly on $\reals\times U$.
	\end{theorem}

\begin{remark}\label{rmk:hyp}
	The assumption that the orbit is hyperbolic is essential in Theorem~\ref{thm:ballcrossing}. For instance, there are "non hyperbolic" examples where the ball-crossing energy can get arbitrary small for arbitrarily large iterations $k$; see~\cite[Remark~3.4]{GG:hyp12}.
\end{remark}
	
	\subsection{Proof of Theorem~\ref{maintheo}}
	Let $\zeta$ be the free homotopy class of the hyperbolic periodic orbit $\gamma$. By passing, if necessary, to the second iteration assume 
	\[
	d\phi\colon T_{\gamma(0)}M \rightarrow T_{\gamma(0)}M
	\]
	has an even number of real eigenvalues in $(-1,0)$. Then there exists a trivialization of $TM$ along $\gamma$ such that the mean index of $\gamma$ (and, since $\gamma$ is hyperbolic, the Conley--Zehnder index) is equal to zero. Moreover, add a constant to the associated hamiltonian $H_t$ (see \eqref{eqn:Ham}) so that the action of some lift of $\gamma$ is
	\begin{eqnarray}\label{eqn:action_gamma}
	\CA_{\{\phi_t\}}(\overline{\gamma})=0.
	\end{eqnarray} 
	
	We reason by contradiction and suppose $\phi$ has finitely many periodic points $x_1,\ldots,x_m$. 
	
	Consider an almost complex structure $J^{'}$ on $M$.  Let $U$ be a small closed neighborhood of $\gamma$ such that $\phi$ has no periodic orbit intersecting $U$ except $\gamma$. By the Ball-Crossing Energy Theorem, there exists a constant $c_{\infty}>0$ such that, for all $k$, the energy of any solution of \eqref{eqn:gradientflowline} of period $k$ asymptotic to $\gamma^k$ as $s\rightarrow\infty$ is greater than $c_{\infty}$.\\
	
	For each $i=1,\ldots,m$, fix a reference loop $z_i$ in $\zeta_i$ (e.g. take $z_i=x_i$) and consider a lift $\overline{x_i}:=[x_i, v_i]\in \widetilde{\mathcal{L}}_{\zeta_i}M$ of $x_i\in \mathcal{L}_{\zeta_i}M$, where $\zeta_i$ is the free homotopy class of $x_i$ and $v_i$ is a cylinder connecting $x_i$ to $z_i$. We have that
	\[
	a_i:=\mathcal{A}_{\{\phi_t\}}(\overline{x_i}) \; (\text{mod } h_{\varrho})\; \text{in}\;  S^1_{h_{\varrho}}\quad\quad\text{and}\quad\quad
	\underline{a_i}:=\underline{\mathcal{A}}_{\{\phi_t\}}(\overline{x_i})\] are independent of the considered lift of $x_i$. 
	
	\begin{remark}
	Observe that, for all $i=1,\ldots,m$, the difference between the actions of two lifts of $x_i$ is a multiple of $h_{\zeta_i}$ when $\varrho$ is rational in the sense of \eqref{eqn:varphi_rational}. Since, for all $i=1,\ldots,m,$ the constant $h_{\varrho}$ divides $h_{\zeta_i}$ (recall \eqref{eqn:h_zeta}) it follows that the difference between these action values is also a multiple of $h_{\varrho}$ and hence $\CA_{\phi_t}(\overline{x_i})$ (mod $h_{\varrho}$) is independent of the considered lift of $x_i$. The fact that the second expression is independent of the lift of $x_i$ follows from Remark~\ref{rmk:augmented_independent}.
	\end{remark}

	Take $\epsilon,\;\delta>0$ sufficiently small so that
	\begin{eqnarray}\label{eqn:epsdelt}
	2(\epsilon+\delta)<\lambda_{\varrho} \text{ and } \epsilon<c_{\infty}
	\end{eqnarray}
	and $C$ a sufficiently large constant  so that
	\begin{eqnarray}\label{eqn:constant_C}
	C> \nu_{T}h_{\varrho}+\frac{\lambda_{\varrho}}{2}(n+1).
	\end{eqnarray}
	Then, by Kronecker's Theorem (see e.g. \cite[Theorem~7.10]{Apostol90}), there exists $k$ such that for all $i=1,\ldots,m$
	\begin{eqnarray}\label{eqn:actioneps}
	||k a_i||_{h_{\varrho}}<\epsilon\quad
	\end{eqnarray} and \begin{eqnarray}\label{eqn:actionC} \text{either } |\underline{a_i}|=0 \text{ or } |k\underline{a_i}|>C.
	\end{eqnarray}
	where we denote by $||a||_{h}\in [0,h / 2]$ the distance from $a\in S^1_{h}=\reals/{h\Z}$ to $0$. Observe that $k$ depends on $\epsilon$ (and $\delta$) and $C$, hence it also depends on $c_{\infty}$ and also on the fixed neighborhood $U$.
	
	Consider a non-degenerate perturbation $\phi{'}$ of $\phi^k$ such that (\ref{eqn:homol_deg}) holds and such that the Hamiltonian $K$ associated with $\phi{'}$ (in the sense of \eqref{eqn:Ham}) satisfies the following properties:
	
	\begin{enumerate}\label{enumerate:K}
		\item \label{enumerate:K2}$K$ coincides with $H^{\natural k}$ on the neighborhood $U$,
		\item \label{enumerate:K3}$K$ is $k$-periodic and non-degenerate and
		\item \label{enumerate:K1}$K$ is sufficiently $C^2$-close to $H^{\natural k}$.
	\end{enumerate}
	Here, we consider $K$ sufficiently $C^2$-close to $H^{\natural k}$ in order to have the existence of $k$ such that for all $x$ $k$-periodic solution of $K$ 
	\begin{eqnarray}\label{eqn:action_epsbounds}
	\big\|\mathcal{A}^{h_{\varrho}}_{\{\phi_t^{'}\}}(\overline{x})\big\|_{h_{\varrho}}<\epsilon
	\end{eqnarray}
	and
	\begin{eqnarray}\label{eqn:augmented_bounds}
	\text{either } \big|\underline{\mathcal{A}}_{\{\phi_t^{'}\}}(\overline{x})\big|<\delta
	\text{ or } \big|\underline{\mathcal{A}}_{\{\phi_t^{'}\}}(\overline{x})\big|>C
	\end{eqnarray}
	where $\mathcal{A}^{h_{\varrho}}_{\{\phi_t^{'}\}}(\overline{x})$ stands for $\mathcal{A}_{\{\phi_t^{'}\}}(\overline{x})$ mod $h_{\varrho}$. Note that, as long as $\delta<h_{\varrho},$ conditions \eqref{eqn:action_epsbounds} and \eqref{eqn:augmented_bounds} follow from \eqref{eqn:actioneps} and \eqref{eqn:actionC}, respectively. Observe that if $\phi^k$ is non-degenerate, then we can take $\phi'=\phi^k$. 
	
	For any $k$-periodic almost complex structure $J$ sufficiently close to (the $k$-periodic extension of) $J'$, all non-trivial $k$-periodic solutions of the equation \eqref{eqn:gradientflowline} for the pair ($\phi'$, $J$) asymptotic to $\gamma^k$ as $s\rightarrow\infty$ have energy greater than $c_{\infty}$.
	
	\begin{lemma}\emph{{\cite[Lemma~4.1]{GG:hyp12}}}\label{lemma:tau}
		Let $\tau := C - \frac{\lambda_{\varrho}}{2}(n+1)$. The orbit ${\overline{\gamma}}^k$ is not connected by a solution of \eqref{eqn:gradientflowline} to any $\overline{x}\in \overline{\mathcal{P}_{k\zeta}}$  with Conley--Zehnder index $\pm 1$ and action in $(-\tau,\tau)$, where $\zeta$ is the free homotopy class of $\gamma$.
		
		In particular, ${\overline{\gamma}}^k$ is closed in $C^{(-\tau,\tau)}_*$ and $0\not=[{\overline{\gamma}}^k]\in HFN^{(-\tau,\tau)}_*(\theta^{\prime}_t, k\zeta)$. Moreover, ${\overline{\gamma}}^k$ must enter every cycle representing its homology class $[{\overline{\gamma}}^k]$ in $HFN^{(-\tau,\tau)}_*(\theta^{\prime}_t, k\zeta)$.
	\end{lemma}
	\begin{proof}
	Assume the orbit ${\overline{\gamma}}^k$ is connected, by a solution $u$ of \eqref{eqn:gradientflowline}, to some $\overline{x}\in \overline{P_{k\zeta}}$ with index $\MUCZ(\overline{x})=\pm 1$ with action in $(-\tau,\tau)$.
	
	Consider the first case in \eqref{eqn:augmented_bounds}, i.e. $\big|\underline{\mathcal{A}}_{\{\phi_t^{'}\}}(\overline{x})\big|<\delta$: since
	\begin{itemize}
		\item[i)] $\big\|\mathcal{A}^{h_{\varrho}}_{\{\phi_t^{'}\}}(\overline{x})\big\|_{h_{\varrho}}<\epsilon$ (by \eqref{eqn:action_epsbounds}),
		\item[ii)]$E(u)>c_{\infty}>\epsilon$ (by Theorem~\ref{thm:ballcrossing} and \eqref{eqn:epsdelt}) and
		\item[iii)]$\mathcal{A}_{\{\phi_t^{'}\}}({\overline{\gamma}}^k)=0$ (by \eqref{eqn:action_gamma}),
	\end{itemize}
	we have
	\[
	\big|\mathcal{A}_{\{\phi_t^{'}\}}(\overline{x})\big|>h_{\varrho} - \epsilon.
	\]
	Then, by the definition of augmented action \eqref{eqn:augmentedaction} and since
	\begin{itemize}
		\item[i)] $\big|\underline{\mathcal{A}}_{\{\phi_t^{'}\}}(\overline{x})\big|<\delta$ and
		\item[ii)] $2(\epsilon +\delta)<\lambda_{\varrho}$ (by \eqref{eqn:epsdelt}),
	\end{itemize}
	we have
	\[
	\big|\D_{\{\phi_t^{'}\}}(\overline{x})\big|> \frac{2}{\lambda_{\varrho}}(h_{\zeta}-\epsilon-\delta)\geq2c_{1,T}^{\min}-\frac{2(\epsilon+\delta)}{\lambda_{\varrho}} >2c_{1,T}^{\min}-1.
	\]
	The second inequality follows from $h_{\zeta}/\lambda_{\varrho}=c_{1,\zeta}^{\min}$ and $c_{1,\zeta}^{\min}\geq c_{1,T}^{\min}.$
	Thus, by \eqref{eqn:mi_czi},
	\[
	\big|\MUCZ(\overline{x})\big|>2c_{1,T}^{\min}-1-n\geq n+2-1-n=1
	\]
	where the second inequality follows from the requirement that $c_{1,T}^{\min}\geq n/2 +1$. We obtained a contradiction since $\MUCZ(\overline{x})=\pm 1$.
	
	Consider now the second case in \eqref{eqn:augmented_bounds}, i.e. $\big|\underline{\mathcal{A}}_{\{\phi_t^{'}\}}(\overline{x})\big|>C$: by the definition of augmented action \eqref{eqn:augmentedaction}, we obtain
	\[
	\big|\mathcal{A}_{\{\phi_t^{'}\}}(\overline{x})\big|> C - \frac{\lambda_{\varrho}}{2} \big|\D_{\{\phi_t^{'}\}}(\overline{x})\big| >
	C-\frac{\lambda_{\varrho}}{2}(n+1)=:\tau
	\]
	where the second inequality follows from the fact that $\big|\D_{\{\phi_t\}}(\overline{x})\big|<n+1$ (which holds since $\MUCZ(\overline{x})=\pm 1$ and by \eqref{eqn:mi_czi}). Hence the action of $\overline{x}$ is outside the interval $(-\tau,\tau)$ and we obtained a contradiction. 
	\end{proof}
	The previous lemma also holds for $q{\overline{\gamma}}^k$ (where $q\overline{x}$ is as defined in Section~\ref{section:capproduct}) with the shifted range of actions $(-\tau,\tau)-\nu_{T}h_{\varrho}.$
	For an interval $(a,b)$ containing the interval $[-\nu_{T}h_{\varrho},0]$ and contained in the intersection of the action intervals $(-\tau,\tau)$ and $(-\tau,\tau)-\nu_{T}h_{\varrho}$,
	Lemma~\ref{lemma:tau} holds for both ${\overline{\gamma}}^k$ and $q{\overline{\gamma}}^k$ and the interval $(a,b)$.
	\begin{remark}
		Observe that the existence of such an interval $(a,b)$ is guaranteed by $-\tau<-\nu_{T}h_{\varrho}<0<\tau-\nu_{T}h_{\varrho}$ that follows from \eqref{eqn:constant_C}.
	\end{remark}
	For the sake of completeness, we state this result in the following lemma.
	\begin{lemma}\label{lemma:qgamma}
		The \emph{orbits} ${\overline{\gamma}}^k$ and $q{\overline{\gamma}}^k$ are not connected by a solution of \eqref{eqn:gradientflowline} to any $\overline{x}\in \overline{\mathcal{P}_{k\zeta}}$  with Conley--Zehnder index $\pm 1$ and action in $(a,b)$ where
		\[ [-\nu_{T}h_{\varrho},0]\subset (a,b)\subseteq (-\tau,\tau)\cap (-\tau-\nu_{T}h_{\varrho},\tau-\nu_{T}h_{\varrho}).\]
		
		In particular, ${\overline{\gamma}}^k$ and $q{\overline{\gamma}}^k$ are closed in $C^{(a,b)}_*$ and $[{\overline{\gamma}}^k]\not=0\not=[q{\overline{\gamma}}^k]\in HFN^{(a,b)}_*(\theta^{\prime}_t, k\zeta)$. 
		
		Moreover, the orbits ${\overline{\gamma}}^k$ and $q{\overline{\gamma}}^k$ must enter every cycle representing their homology classes, respectively $[{\overline{\gamma}}^k]$ and $[q{\overline{\gamma}}^k]$, in $HFN^{(a,b)}_*(\theta^{\prime}_t, k\zeta)$.
	\end{lemma}
	
	Recall that, by \eqref{eqn:action_epsbounds}, all periodic orbits of $\phi'$ have action values in the $\epsilon$-neighborhood of $h_{\varrho}\Z$. The next lemma yields a contradiction and the main theorem follows.
	
	\begin{lemma}\emph{{\cite[Lemma~4.2]{GG:hyp12}}}\label{lemma:orbity}
		The symplectomorphism $\phi'$ has a periodic orbit with action outside the $\epsilon$-neighborhood of $h_{\varrho}\Z$.
	\end{lemma}
	\begin{proof} For ordinary homology classes $\alpha,\;\beta\in H_*(M)$ with $\deg(\alpha),\;\deg(\beta)<2n$ as in the statement of Theorem~\ref{maintheo}, consider $\Phi_{\beta*\alpha}([\overline{\gamma}^k])$ as an element of the group $HFN_*^{(a,b)}(\theta^{\prime}_t, \zeta)$ with $(a,b)$ as in Lemma~\ref{lemma:qgamma}. Since $\beta*\alpha=q[M]$, it follows, by \eqref{eqn:multip_quantum} and \eqref{diagram:cap}, that
	\[
	\Phi_{\beta}\Phi_{\alpha}([{\overline{\gamma}}^k])=\Phi_{\beta*\alpha}([{\overline{\gamma}}^k])=
	\Phi_{q[M]}([{\overline{\gamma}}^k])=q\Phi_{[M]}([{\overline{\gamma}}^k])= q[{\overline{\gamma}}^k].
	\]
	Take $\sigma$ and $\eta$ generic cycles representing the ordinary homology classes $\alpha$ and $\beta$, respectively. The chain $\Phi_{\eta}\Phi_{\sigma}({\overline{\gamma}}^k)$ represents the homology class $q[\overline{\gamma}^k]$ and hence the \emph{orbit} $q{\overline{\gamma}}^k$ enters the chain $\Phi_{\eta}\Phi_{\sigma}({\overline{\gamma}}^k)$ (by Lemma~\ref{lemma:qgamma}). Recall that $q[x,v]$ is the class $[x,w]$ where $w$ is the cylinder obtained by attaching $\nu_{\zeta} A_{\zeta}$ to $v$. Hence (see Figure~\ref{fig:existence_y}), there exists an \emph{orbit} $\overline{y}$ in $\Phi_{\sigma}({\overline{\gamma}}^k)$ connected to ${\overline{\gamma}}^k$ and $q{\overline{\gamma}}^k$ by trajectories which are solutions of \eqref{eqn:gradientflowline}. By the Ball-Crossing Energy Theorem, \eqref{eqn:epsdelt} and
	\begin{enumerate}
		\item $\mathcal{A}_{\{\phi_t^{'}\}}({\overline{\gamma}}^k)=0$
		\item\label{item:action}$\mathcal{A}_{\{\phi_t^{'}\}}(q{\overline{\gamma}}^k)=-\nu_{T} h_{\varrho}$,
	\end{enumerate}
	we obtain
	\begin{eqnarray}\label{eqn:proof_actionbounds}
	-\epsilon > \mathcal{A}_{\{\phi_t\}}(\overline{y})> -\nu_{T} h_{\varrho} +\epsilon.
	\end{eqnarray}
	The fact in point \ref{item:action} follows from $\varrho(\nu_{k\zeta}A_{k\zeta})=-\nu_{k\zeta}h_{k\zeta}= -\nu_T h_{\varrho}$ where the second equality holds by \eqref{eqn:nus}.
	
	\begin{figure}[htb!]
		\centering
		\def\svgwidth{300pt}
		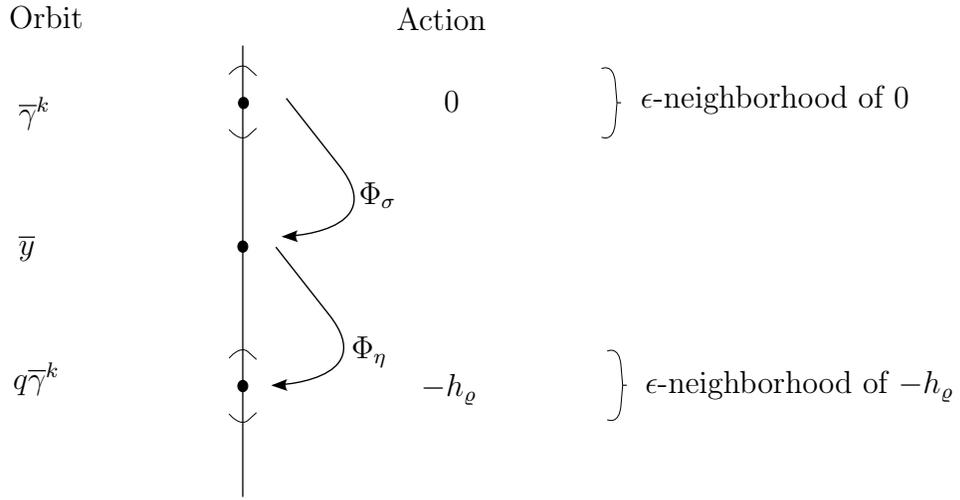
		\caption{The orbit $y$.}\label{fig:existence_y}
	\end{figure}
	
Let us consider the two cases in assumption \eqref{item:qs} of the main theorem. Namely, first assume $c^{\min}_{1}=c^{\min}_{1,T}$ (i.e., 
$\nu_T=1$).	In this case, we obtain, by \eqref{eqn:proof_actionbounds}, that the action value of $\overline{y}$ is outside the $\varepsilon$-neighborhood of $h_{\varrho}\Z$:
		\[
		-\epsilon > \mathcal{A}_{\{\phi'_t\}}(\overline{y}) > -h_{\varrho} +\varepsilon.
		\]

Now, assume $\deg(\alpha)\geq 3n-2c^{\min}_{1,T} +1$. As above, we distinguish the two cases in \eqref{eqn:augmented_bounds}: $\underline{\mathcal{A}}_{\{\phi'_t\}}(\overline{y})<\delta$ and $\underline{\mathcal{A}}_{\{\phi'_t\}}(\overline{y})>C$. Firstly, when $\underline{\mathcal{A}}_{\{\phi'_t\}}(\overline{y})<\delta$, assume $\mathcal{A}_{\{\phi'_t\}}(\overline{y})$ is in the $\eps$-neighborhood of $h_{\varrho}\Z$. Then \eqref{eqn:proof_actionbounds} implies that 
\[\mathcal{A}_{\{\phi'_t\}}(\overline{y})<-h_{\varrho}+\varepsilon.\] 
This condition and the fact that the augmented action is less than 
$\delta$ give the following upper bound for the mean index of $\overline{y}$
	\[\Delta_{\{\phi'_t\}}(\overline{y})<\frac{2}{\lambda_{\varrho}}(-h_{\varrho}+\epsilon+\delta).
	\]
	The right hand side of the inequality is equal to $-2c^{\text{min}}_{1,T}+ \frac{2}{\lambda_{\varrho}}(\varepsilon+\delta)$
	(by \eqref{eqn:c1minTprop}) and, by the choices in \eqref{eqn:epsdelt}, it is less than $-2c^{\text{min}}_{1,T}+1$.

	Moreover, since $\gamma$ is a hyperbolic periodic orbit and $\MUCZ(\overline{y})=\MUCZ(\overline{\gamma}^k)-2n+\deg(\alpha)$, we have 
	\begin{eqnarray} \label{eqn:CZ_alpha}
	\MUCZ(\overline{y}) =-2n+\deg(\alpha)
	\end{eqnarray}
	which implies, by \eqref{eqn:mi_czi}, that $\deg(\alpha)<3n - 2c^{\text{min}}_{1,T}+1$ which contradicts our hypothesis. Before we consider the second case in \eqref{eqn:augmented_bounds}, observe that the mean index of $\overline{y}$ is bounded from above by $n+1$:
	\begin{eqnarray}\label{eqn:Delta}
	\Delta(\overline{y})<n+1.
	\end{eqnarray}
	This follows from \eqref{eqn:mi_czi}, \eqref{eqn:CZ_alpha} 
	and our assumptions on the degree of the homology class $\alpha$. Now, when $\underline{\mathcal{A}}_{\{\phi'_t\}}(\overline{y})>C$, we have $|\mathcal{A}_{\{\phi'_t\}}(\overline{y})|>C-\frac{\lambda_{\varrho}}{2}|\Delta(\overline{y})|$. Then, by \eqref{eqn:constant_C},
	\[
	|\mathcal{A}_{\{\phi'_t\}}(\overline{y})|>C-\frac{\lambda_{\varrho}}{2}|\Delta(\overline{y})|>\nu_{T}h_{\varrho}\]
	which is impossible due to \eqref{eqn:proof_actionbounds}.

\end{proof}

\bibliographystyle{amsalpha}
\bibliography{references}

\vspace{1cm}
\author{Marta Bator\'eo}

\address{Instituto Nacional de Matem\'atica Pura e Aplicada:\\
	\hspace*{0.7cm}Estrada Dona Castorina 110, 
	Rio de Janeiro 22460-320, Brazil}

\email{mbatoreo@impa.br}

\end{document}